\documentclass[11pt]{amsart}
\usepackage[utf8]{inputenc}

\usepackage{amsthm,amssymb}
\usepackage{mathtools}
\usepackage[colorlinks]{hyperref}
\usepackage[margin=1in]{geometry}
\usepackage{tikz}
\usepackage{multirow}
\usepackage{float}
\usepackage{dsfont}
\usepackage{booktabs}
\usepackage{comment}
\usepackage{kotex}

\usepackage{tikz-cd}
\usepackage{tikz}
\usetikzlibrary{3d, matrix, decorations.pathreplacing, calc, decorations.pathmorphing}
\usetikzlibrary {positioning}
\usetikzlibrary{matrix, fit}
\usetikzlibrary{backgrounds}

\renewcommand{\tilde}{\widetilde}

\renewcommand{\hat}{\widehat}

\newcommand{\C}{\mathbb{C}}

\newcommand{\R}{\mathbb{R}}

\newcommand{\Z}{\mathbb{Z}}

\newcommand{\CP}{\mathbb{C}P}
\newcommand{\vlambda}{\boldsymbol{\lambda}}

\newcommand{\ve}{\mathbf{e}}
\newcommand{\vv}{\mathbf{v}}
\newcommand{\vw}{\mathbf{w}}

\newcommand{\vs}{\vspace}

\newcommand{\ds}{\displaystyle}
\newcommand{\B}{\mathcal{B}}

\newcommand{\T}{\mathbb{T}}

\DeclareMathOperator{\Cone}{Cone}
\renewcommand{\P}{\mathcal{P}}

\newcommand{\1}{\mathds{1}}

\def\mathscr{\mathscr}

\def\mcal{\mathcal}

\newtheorem{theorem}{Theorem}[section]

\newtheorem{lemma}[theorem]{Lemma}
\newtheorem{proposition}[theorem]{Proposition}

\newtheorem{corollary}[theorem]{Corollary}

\newtheorem{conjecture}[theorem]{Conjecture}
\newtheorem{question}[theorem]{Question}

\theoremstyle{definition}
\newtheorem{example}[theorem]{Example}
\newtheorem{definition}[theorem]{Definition}
\newtheorem{remark}[theorem]{Remark}

\numberwithin{equation}{section}

\makeatletter
\renewcommand*\env@matrix[1][*\c@MaxMatrixCols c]{%
	\hskip -\arraycolsep
	\let\@ifnextchar\new@ifnextchar
	\array{#1}}
\makeatother

\hypersetup{colorlinks}
\hypersetup{
	unicode=false,          
	pdftoolbar=true,        
	pdfmenubar=true,        
	pdffitwindow=false,     
	pdfstartview={FitH},    
	pdftitle={My title},    
	pdfauthor={Author},     
	pdfsubject={Subject},   
	pdfcreator={Creator},   
	pdfproducer={Producer}, 
	pdfkeywords={keyword1} {key2} {key3}, 
	pdfnewwindow=true,      
	colorlinks=true,       
	linkcolor=blue,          
	citecolor=red,        
	filecolor=violet,      
	urlcolor=violet       
}

\newcommand{\bfa}[2]{{\mathbf{a}_{{#1},{#2}}}}
\newcommand{\aij}[3]{{a_{{#1},{#2}}^{{#3}}}}
\newcommand{\Lij}[3]{{\lambda_{{#1},{#2}}^{{#3}}}}

\DeclareMathOperator{\PC}{PC}

\begin{document}
\title{$c_1$-cohomological rigidity for smooth toric Fano varieties of Picard number two}

\dedicatory{Dedicated to Professor Victor M. Buchstaber on his 80th birthday}

\begin{abstract} 
The $c_1$-cohomological rigidity conjecture states that two smooth toric Fano varieties are isomorphic as varieties if there is a $c_1$-preserving isomorphism between their integral cohomology rings. In this paper, we confirm the conjecture for smooth toric Fano varieties of Picard number two. 
\end{abstract}

\author{Yunhyung Cho}
\address{Department of Mathematics Education, Sungkyunkwan University, Seoul, Republic of Korea}
\email{yunhyung@skku.edu}

\author{Eunjeong Lee}
\address{Center for Geometry and Physics, Institute for Basic Science (IBS), Pohang 37673, Korea}
\email{eunjeong.lee@ibs.re.kr}

\author{Mikiya Masuda}
\address{Osaka Central Advanced Mathematical Institute, Osaka Metropolitan University, Sugimoto, Sumiyoshi-ku, Osaka, 558-8585, Japan}
\email{masuda@osaka-cu.ac.jp}

\author{Seonjeong Park${}^\ast$}
\address{Department of Mathematics Education, Jeonju University, Jeonju 55069, Republic of Korea}
\email{seonjeongpark@jj.ac.kr}

\thanks{
${}^{\ast}$ S. Park is the corresponding author.}

\maketitle

\tableofcontents

\section{Introduction}

Motivated by McDuff's question mentioned later, we posed the following conjecture in \cite{CLMP}. 

\begin{conjecture}[{\cite[Conjecture 1.4]{CLMP}}] \label{conj_Fano_rigidity} 
Let $X$ and $Y$ be smooth toric Fano varieties. If there exists a $c_1$-preserving graded ring isomorphism between their integral cohomology rings, then $X$ and $Y$ are isomorphic as varieties, where $c_1$-preserving means preserving the first Chern classes of $X$ and~$Y$. 
\end{conjecture}

Neither the Fano condition nor the $c_1$-preserving condition can be dropped as is observed for Hirzebruch surfaces. We say that a smooth toric Fano variety $X$ is {\em $c_1$-cohomologically rigid} if any smooth toric Fano variety $Y$ which allows a $c_1$-preserving graded ring isomorphism $H^*(Y;\Z)\to H^*(X;\Z)$ is isomorphic to $X$ as a variety. Then Conjecture \ref{conj_Fano_rigidity} is equivalent to saying that every smooth toric Fano variety is $c_1$-cohomologically rigid. 

The $c_1$-cohomological rigidity is verified for Fano Bott manifolds (\cite{CLMP}), smooth toric Fano varieties of dimension up to four or of Picard number greater than or equal to $2n-2$ (\cite{HKM22}), where $n$ is the complex dimension of the Fano variety. In each dimension $n$, the complex projective space $\C P^n$ is the only smooth compact toric variety of Picard number one, which is Fano. However, there are many smooth toric Fano varieties of Picard number two. In this paper, we prove the following. 

\begin{theorem} \label{thm_rank_2}
Every smooth toric Fano variety of Picard number two is $c_1$-cohomologically rigid. 
\end{theorem}

In fact, a smooth compact toric variety $X$ of Picard number two is a two-stage generalized Bott manifold, that is, the total space of a $\C P^{n_2}$-bundle over $\C P^{n_1}$ obtained as 
\[
X=P(\underline{\C}\oplus\gamma^{a_1}\oplus\cdots \oplus \gamma^{a_{n_2}})
\]
where $\underline{\C}$ is the trivial line bundle over $\C P^{n_1}$, $\gamma$ is the tautological line bundle over $\C P^{n_1}$, $a_1,\dots,a_{n_2}$ are integers, and $P(\ )$ denotes projectivization. We may assume that $a_1,\dots,a_{n_2}$ are nonnegative without loss of generality and then $X$ is Fano if and only if $\sum_{k=1}^{n_2}a_k\le n_1$. 

Conjecture \ref{conj_Fano_rigidity} is closely related to an algebraic property of the group of 
Hamiltonian diffeomorphisms of a monotone (symplectic) toric manifold. In \cite{McDuff2011}, McDuff raised the following question.

\begin{question}[{McDuff, \cite[Question~1.11]{McDuff2011}}] \label{Q_1}
Is there a monotone toric manifold $(M,\omega)$ with more than one toric structure?
\end{question}

Here, a symplectic form $\omega$ is called {\em monotone} if $c_1(M) := c_1(TM,J) = \lambda \cdot [\omega]$ for some $\lambda > 0$ and an $\omega$-compatible almost complex structure $J$ on $M$. McDuff's question asks whether compact $n$-tori in the group $\mathrm{Ham}(M,\omega)$ of all Hamiltonian diffeomorphisms of a monotone toric manifold $(M,\omega)$ are conjugate to each other in $\mathrm{Ham}(M,\omega)$, where $n=\frac{1}{2}\dim_{\R} M$. We may take $\lambda = 1$ because $\mathrm{Ham}(M,\omega)$ is independent of $\lambda$. 

A monotone toric manifold $(M,\omega)$ (with $\lambda=1$) is known to be equivariantly symplectomorphic to a smooth toric Fano variety $X$, where the symplectic form on $X$ is a first Chern form of $X$. Indeed, this correspondence is one-to-one (up to isomorphism) because monotone toric manifolds bijectively correspond to the dual of smooth Fano polytopes through moment maps while smooth toric Fano varieties bijectively correspond to smooth Fano polytopes through fans. 

The following proposition shows that an affirmative solution of Conjecture \ref{conj_Fano_rigidity} answers Question~\ref{Q_1}. 

\begin{proposition} \label{prop_rigidity_implies_uniqueness}
If the smooth toric Fano variety $X$ corresponding to a monotone toric manifold $(M,\omega)$ is $c_1$-cohomologically rigid, then $(M,\omega)$ has a unique toric structure. 
\end{proposition}

\begin{proof}
Let $T_1$ and $T_2$ be two $n$-tori in $\mathrm{Ham}(M,\omega)$ where $n=\frac{1}{2}\dim_{\R} M$. Then, monotone toric manifolds $(M,\omega,T_1)$ and $(M,\omega,T_2)$ can be regarded as smooth toric Fano varieties and there is an obvious $c_1$-preserving cohomology ring isomorphism between them, that is the identity. Therefore, $(M,\omega,T_1)$ and $(M,\omega,T_2)$ are isomorphic as varieties since $X$ is $c_1$-cohomologically rigid. Then they are isomorphic as toric varieties up to an automorphism of the algebraic torus acting on them\footnote{In general, if two smooth compact toric varieties $X_1, X_2$ are isomorphic as varieties, then there exists an isomorphism $f\colon X_1\to X_2$ together with an automorphism of the algebraic torus $\T$ acting on $X_1$ and $X_2$ such that $f(gx)=\sigma(g)f(x)$ for $g\in \T$ and $x\in X$, see Section~\ref{subsection_equivariant_cohomologies} in this paper for details.}. 
This implies that there is a $\Z$-linear isomorphism between the fans, and hence between the moment polytopes for $(M,\omega,T_1)$ and $(M,\omega,T_2)$ as we fixed $\lambda = 1$.
Therefore, by Delzant's theorem~\cite{Del}, there is a symplectomorphism $f\colon (M,\omega,T_1)\to (M,\omega,T_2)$ together with an isomorphism $\sigma\colon T_1\to T_2$ such that $f(gp)=\sigma(g)f(p)$ for $g\in T_1$ and $p\in M$. This means that $f$ is in $\mathrm{Ham}(M,\omega)$ and $T_1$ is conjugate to $T_2$ in $\mathrm{Ham}(M,\omega)$ by $f$. 
\end{proof}

Thus, we obtain the following as a corollary of Theorem~\ref{thm_rank_2}.

\begin{corollary}[{\cite[Corollary 1.15]{fano14}}]
The monotone toric manifold associated with a smooth toric Fano variety of Picard number two has a unique toric structure. 
\end{corollary}

This paper is organized as follows. In Section~\ref{sec:generalized_Bott}, we briefly review generalized Bott manifolds and the presentation of their cohomology ring in terms of so-called {generalized Bott matrices}. In Section~\ref{secFanoGeneralizedBottManifolds}, we investigate when generalized Bott manifolds are Fano by applying Batyrev's criterion. In Section~\ref{sec:two-stage}, we prove Theorem~\ref{thm_rank_2}. In Section~\ref{sec:concluding_remarks}, we discuss related cohomological rigidity problems and results. 

\subsection*{Acknowledgements}
Y. Cho was supported by the National Research Foundation of Korea(NRF) grant funded by the Korea government(MSIP; Ministry of Science, ICT \& Future Planning) (No.\ 2020R1C1C1A01010972) and (No.\ 2020R1A5A1016126). E. Lee was supported by the National Research Foundation of Korea(NRF) grant funded by the Korea government(MSIT) (No.\ RS-2023-00239947).
M. Masuda was supported in part by JSPS Grant-in-Aid for Scientific Research 22K03292.  
S. Park was supported by the National Research Foundation of Korea [NRF-2020R1A2C1A01011045].
This work was partly supported by Osaka Central Advanced Mathematical Institute (MEXT Joint Usage/Research Center on Mathematics and Theoretical Physics) and the HSE University Basic Research Program.

\section{Generalized Bott manifolds}
\label{sec:generalized_Bott}

In this section, we recall some basic facts on generalized Bott manifolds from~\cite{CMS-tams}.
\begin{definition}[\cite{CMS-tams}]
A \emph{generalized Bott tower} $\B_\bullet$ of height $m$ is an iterated $\C P^{n_{i}}$-bundle:
\begin{equation*}
\begin{tikzcd}[row sep = 0.2cm]
\mathcal{B}_{m} \arrow[r, "\pi_{m}"] \arrow[d, equal]& 
\mathcal{B}_{m-1} \arrow[r, "\pi_{m-1}"] &
\cdots \arrow[r, "\pi_2"] &
\mathcal{B}_1 \arrow[r, "\pi_1"] \arrow[d, equal]&
\mathcal{B}_0, \arrow[d, equal]\\
P(\underline{\C}\oplus \xi_{m-1}^{1}\oplus \cdots \oplus \xi_{m-1}^{n_{m}}) & & & \C P^{n_{1}} & \{\text{a point}\}
\end{tikzcd} 
\end{equation*}
where each $\mathcal{B}_{i}$ is the complex projectivization of the Whitney sum of line bundles~$\xi_{i-1}^{k}$ $(1\leq k\leq n_{i})$ and the trivial line bundle $\underline{\C}$ over $\mathcal{B}_{i-1}$.  
We call $\B_m$ an $m$-stage {\em generalized Bott manifold}. When $n_i=1$ for every $i$, a generalized Bott tower is called a \emph{Bott tower}, and accordingly a generalized Bott manifold is called a \emph{Bott manifold}. 
\end{definition}
The fiber of the bundle $\pi_i\colon \B_{i}\to \B_{i-1}$ is the complex projective space $\C P^{n_i}$ and if the line bundles $\xi_{i-1}^{k}$ constructing the tower $\B_\bullet$ are all trivial, then $\B_m$ is isomorphic to $\prod_{i=1}^m \C P^{n_i}$ as a variety. 

Let $\gamma_j$ be the tautological line bundle over $\mathcal{B}_j$. By abuse of notation, we denote the pullback of~$\gamma_j$ by the projection $\pi_i\circ \cdots \circ\pi_{j+1}\colon \mathcal{B}_{i}\to \mathcal{B}_j$ for $i>j$ by the same notation $\gamma_j$. Then the Picard group $\mathrm{Pic}(\B_{i})$ of $\B_{i}$ is generated by the line bundles $\gamma_{j}$ for $1 \leq j \leq i$ and isomorphic to $\Z^{i}$. Thus each line bundle $\xi_{i-1}^k$ (where $1 \leq k \leq n_i$) over $\B_{i-1}$ can be expressed by 
\begin{equation}\label{eq_xi_and_aij}
\xi_{i-1}^{k}=\bigotimes_{1\leq j < i}\gamma_{j}^{\otimes a_{i,j}^{k}}
\end{equation}
for some integers $a_{i,j}^{k} \in \Z$ with $1 \leq j < i$. Accordingly, the set $\{ a_{i,j}^{k} \}_{\substack{1 \leq j < i \leq m, \\ 1\leq k\leq n_{i}}}$ of integers determines a generalized Bott manifold.

The projection map $\pi_{i} \colon \B_{i} \rightarrow \B_{i-1}$ admits a section induced from the zero section of the vector bundle $\underline{\C} \oplus \bigoplus_{k=1}^{n_{i}} \xi_{i-1}^{k}$ so that 
the induced ring homomorphism 
\[
\pi_i^* \colon H^*(\B_{i-1};\Z) \rightarrow H^*(\B_i;\Z)
\]
is injective and we think of elements in $H^*(\B_j;\Z)$ as elements in $H^*(\B_m)$ for any $j\le m$. We set 
\[
x_j:=-c_1(\gamma_j)\in H^2(\B_m;\Z).
\]
Then it follows from \eqref{eq_xi_and_aij} that 
\[
c_1(\xi_{i-1}^{k})=-\sum_{j=1}^{i-1}a_{i,j}^{k}x_j\in H^2(\B_{m};\Z).
\]

A generalized Bott manifold $\B_m$ is a smooth projective toric variety of $\C$-dimension $n \coloneqq \sum_{i=1}^m n_i$, where the algebraic torus action can be constructed in an iterative way using a toric structure of a base space and a $(\C^*)^{n_i}$-action on a fiber at each stage. The associated fan $\Sigma$ can be described as follows. (See also ~\cite[\S 7.3]{CLStoric11}.) Let $\{ \ve^1_1,\dots,\ve^{n_1}_1,\dots,\ve^1_m,\dots,\ve^{n_m}_m\}$ be the standard basis vectors of $\R^n = \R^{n_1} \oplus \cdots \oplus \R^{n_m}$. Then the set of ray generators of $\Sigma$ is given by the columns of the matrix
\begin{equation}\label{eq_Bott_matrix}
(E~|~A)\coloneqq\begin{pmatrix}
E_{n_{1}}& & & & & -\1 & & & & \\
& E_{n_{2}}& & & & \mathbf{a}_{2,1} & -\1 & & & \\
& &E_{n_{3}}& & & \mathbf{a}_{3,1} & \mathbf{a}_{3,2} & -\1 & & \\
& & & \ddots & &\vdots & \vdots & \ddots &\ddots & \\
& & & &E_{n_{m}}& \mathbf{a}_{m,1} & \mathbf{a}_{m,2} &\cdots & \mathbf{a}_{m,m-1} & -\1\\
\end{pmatrix},
\end{equation}
where $E_{n_j}$ is the identity matrix of size $n_j$, i.e., the column vectors of $E_{n_j}$ are $\ve^{1}_j,\dots,\ve^{n_j}_j$, and 
\[
\bfa{i}{j} \coloneqq (\aij{i}{j}{1},\dots,\aij{i}{j}{n_i})^T \in \Z^{n_i} \qquad\text{ and }\qquad {-\1} = (-1,\dots,-1)^T
\] for $j=1,\dots,m$. For simplicity, we denote the $(n+j)$th column vector of~\eqref{eq_Bott_matrix} by $\vv_j$ for $j=1,\dots,m$. Then it
is easy to see that $\P$ is a maximal cone in $\Sigma$ if and only if $\P$ is of the form
\begin{equation*}
\P = \Cone(\hat{\P}_1 \cup \dots \cup \hat{\P}_m),
\end{equation*}
where 
\begin{equation*}
R_j = \{ \ve_{j}^{1},\dots,\ve_{j}^{n_{j}},\vv_{j}\},\quad \hat{\P}_j = R_j \backslash\{ \vw_j \} \quad \text{ for some } \vw_j \in R_j, j=1,\dots,m.
\end{equation*} 
In particular, $\Sigma$ is combinatorially equivalent to the product fan $\Sigma_1 \times \cdots \times \Sigma_m$, where $\Sigma_i$ is the fan of $\C P^{n_i}$ and so there are $\prod_{j=1}^m (n_j+1)$ maximal cones. We call the matrix $A=[\vv_1 \dots \vv_m]$ a {\em generalized Bott matrix} of type $(n_1,\dots,n_m)$. 

\begin{example} 
A generalized Bott matrix $A$ of type $(1,4,2)$ has the following form:
\[
A = \begin{bmatrix}
- \1 & \mathbf{0} & \mathbf{0} \\ 
\bfa{2}{1} & - \1 & \mathbf{0}\\
\bfa{3}{1} & \bfa{3}{2} & -\1 
\end{bmatrix} = 
\begin{bmatrix}
-1 & 0 & 0 \\
\aij{2}{1}{1} & -1 & 0 \\
\aij{2}{1}{2} & -1 & 0 \\
\aij{2}{1}{3} & -1 & 0 \\
\aij{2}{1}{4} & -1 & 0\\
\aij{3}{1}{1} & \aij{3}{2}{1} & -1 \\
\aij{3}{1}{2} & \aij{3}{2}{2} & -1
\end{bmatrix}.
\]
\end{example} 

We set
\[
\alpha_{i}^{k}\coloneqq a_{i,1}^{k}x_{1}+\dots+a_{i,i-1}^{k}x_{i-1}\in H^{2}(\B_m;\Z) \qquad (i=2,\dots,m,\quad k=1,\dots,n_{i}).
\]
By the Borel--Hirzebruch formula, the integral cohomology ring of $\B_m$ can be represented by 
\begin{equation}\label{eq_coh_GB}
\begin{split}
H^\ast(\B_m;\Z)&= \Z[x_1,\dots,x_m] \Bigg/ \left\langle x_1^{n_1+1},\ x_i\prod_{k=1}^{n_{i}}(x_{i}-\alpha_{i}^{k}) \ \ (i=2,\dots,m) \right\rangle,
\end{split}
\end{equation}
where $\langle\ \rangle$ denotes the ideal generated by the elements in it. The total Chern class of $\B_m$ is written by
\[
\begin{array}{ccl}\vs{0.2cm}
c(\B_m) & = & \ds (1+x_1)^{n_1+1}\prod_{i=2}^{m}\left[(1+x_{i})\prod_{k=1}^{n_{i}}(1+x_{i}-\alpha_{i}^{k})\right]
\end{array}
\]
and in particular, we have 
\begin{equation} \label{eq:c1}
\begin{array}{ccl}
c_1(\B_m) & = &\ds (n_1+1)x_1+\sum_{i=2}^{m}\left\{(n_{i}+1)x_i-\sum_{k=1}^{n_{i}}\alpha_i^k\right\}\\
&= &\ds \sum_{i=1}^{m}(n_{i}+1)x_i-\sum_{i=2}^m\sum_{k=1}^{n_{i}}\alpha_i^k.
\end{array}
\end{equation}

\begin{remark}\label{rm1} 
\begin{enumerate}
\item We obtain $H^*(\B_m;\Z)\cong H^*(\prod_{i=1}^m\C P^{n_i};\Z)$ as \emph{groups}, so $H^*(\B_m)$ determines the multiset $\{n_1,\dots,n_m\}$ of fiber dimensions in the generalized Bott tower $\B_\bullet$. 
\item A smooth compact toric variety whose cohomology ring is isomorphic to the cohomology ring of a generalized Bott manifold is a generalized Bott manifold. (See \cite{CMS-tams} or \cite{CPS-lms}.)
\end{enumerate}
\end{remark}

\section{Fano condition}
\label{secFanoGeneralizedBottManifolds}

In this section, we describe the Fano condition on generalized Bott manifolds using Batyrev's criterion in ~\cite{Batyrev91, Batyrev99}. We refer the reader to~\cite{Suyama20} for another description of the Fano condition on generalized Bott manifolds. 

For a complete nonsingular fan $\Sigma$, a subset $R$ of the primitive ray vectors is called a \emph{primitive collection} of~$\Sigma$ if 
\[
\Cone(R) \notin \Sigma
\quad \text{ but }\quad \Cone(R \setminus \{\mathbf{u}\}) \in \Sigma \quad \text{ for every }\mathbf{u} \in R.
\] 
We denote by $\PC(\Sigma)$ the set of primitive collections of $\Sigma$.
For a primitive collection $R = \{\mathbf{u}'_1, \dots,\mathbf{u}'_{\ell}\}$, we have $\mathbf{u}'_1 + \cdots+\mathbf{u}'_{\ell}=\boldsymbol{0}$ or there exists a unique cone~$\sigma$ of positive dimension such that $\mathbf{u}'_1 + \cdots+\mathbf{u}'_{\ell}$ is in the interior of $\sigma$. That is, 
\begin{equation}\label{eq:primitive}
\mathbf{u}'_1 + \cdots+\mathbf{u}'_{\ell}=\begin{cases}
\boldsymbol{0}, &\text{ or }\\
\lambda_1 \mathbf{u}_1 + \cdots+ {\lambda_{s}} \mathbf{u}_{s},&{}
\end{cases}
\end{equation}
where $\mathbf{u}_1,\dots,\mathbf{u}_{s}$ are the primitive generators of $\sigma$ and $\lambda_1,\dots,\lambda_{s}$ are positive integers.
We call~\eqref{eq:primitive} a \emph{primitive relation}, and the \emph{degree} $\deg R$ of a primitive collection $R$ is defined to be 
\begin{equation} \label{eq:degree}
\deg R :=
\begin{cases}
\ell & \text{ if }\mathbf{u}'_1 + \cdots+\mathbf{u}'_{\ell} = \mathbf{0}, \\
\ell - (\lambda_1+\cdots+\lambda_s) &\text{ otherwise}. 
\end{cases}
\end{equation}

\begin{proposition}[{\cite[Proposition~2.3.6]{Batyrev99}}]\label{prop:batyrev}
A smooth compact toric variety $X$ is Fano if and only if $\deg R>0$ for every primitive collection $R$ of the fan $\Sigma$ of $X$. 
\end{proposition}

From the description of the fan $\Sigma$ of $\B_m$, we have 
\[
\PC(\Sigma) = \{ R_j \mid j =1,\dots,m\} = \{ \{ \vv_j, \ve_{j}^1,\dots,\ve_{j}^{n_{j}} \} \mid j=1,\dots,m\}.
\]
For each primitive collection $R_j$ $(1\le j<m)$, we obtain nonnegative integer vectors $\vlambda_{i,j}=(\lambda_{i,j}^{0},\lambda_{i,j}^{1},\dots,\lambda_{i,j}^{n_{i}})\in\Z_{\geq 0}^{n_{i}+1}$ $(1\leq j<i\leq m)$ such that 
\begin{align*}
&\vv_j + \ve^{1}_{j}+\dots+\ve^{n_{j}}_{j} =\sum_{i=j+1}^{m}\left(\lambda_{i,j}^{0}\vv_{i}+\sum_{k=1}^{n_{i}}\lambda_{i,j}^{k}\ve^{k}_{i}\right)\quad\text{and}\\
&\prod_{k=0}^{n_{i}}\lambda_{i,j}^{k}=0\quad \text{for each $i$}.
\end{align*}

The following lemma follows immediately from Proposition~\ref{prop:batyrev}. 
\begin{lemma}\label{lem:Fano}
The generalized Bott manifold $\B_m$ is Fano if and only if 
\begin{equation}
\sum_{i=j+1}^{m}\sum_{k=0}^{n_{i}}\lambda_{i,j}^{k}\leq n_{j} \quad \text{ for all }j=1,\dots,m.\label{eq:fano2}
\end{equation}
\end{lemma}

\begin{example}
For a generalized Bott manifold $\B_3$ associated with a generalized Bott matrix of type $(n_1,n_2,n_3)=(3,2,2)$ defined by
\[
A = \begin{bmatrix}
- \1 & \mathbf{0} & \mathbf{0} \\ 
\bfa{2}{1} & - \1 & \mathbf{0}\\
\bfa{3}{1} & \bfa{3}{2} & -\1 
\end{bmatrix} = 
\begin{bmatrix}
-1 & 0 & 0 \\
-1 & 0 & 0 \\
-1 & 0 & 0 \\
-1 & -1 & 0 \\
-1 & -1 & 0 \\
1 & 1 & -1 \\
2 & 0 & -1
\end{bmatrix},
\]
we have 
\[
\begin{split}
\vv_1+\ve^1_1+\ve_1^2 +\ve_1^3 &= (0, 0, 0, -1, -1, 1, 2)^T = \vv_2+2\ve_3^2, \\
\vv_2+\ve^1_2 + \ve_2^2 &= (0,0,0,0,0,1,0)^T = \ve^1_3, \\
\vv_3+\ve^1_3 + \ve^2_3 &= \mathbf{0}.
\end{split}
\]
Therefore, the integer vectors $\vlambda_{i,j}$ are given as follows:
\[
\begin{split}
&\vlambda_{2,1} = (\Lij{2}{1}{0},\Lij{2}{1}{1},\Lij{2}{1}{2}) = (1,0,0), \quad 
\vlambda_{3,1} = (\Lij{3}{1}{0}, \Lij{3}{1}{1}, \Lij{3}{1}{2}) = (0,0,2), \\
& \vlambda_{3,2} = (\Lij{3}{2}{0},\Lij{3}{2}{1},\Lij{3}{2}{2}) = (0, 1,0).
\end{split}
\]
Since we have
\[
\sum_{i=2}^3 \sum_{k=0}^{n_i} \Lij{i}{1}{k} = 3 \leq n_1 = 3 \quad\text{and}\quad
\sum_{k=0}^{n_3} \Lij{3}{2}{k} = 1 \leq n_2 = 2,
\]
the generalized Bott manifold $\B_3$ is Fano by Lemma~\ref{lem:Fano}. 
\end{example}

\section{Two-stage Fano generalized Bott manifolds}\label{sec:two-stage}

It is known that two-stage generalized Bott manifolds are {\em diffeomorphic} to each other if and only if their cohomology rings are isomorphic as graded rings (\cite[Theorem~6.1]{CMS-tams}). There are many two-stage generalized Bott manifolds which are diffeomorphic but not isomorphic as varieties to each other. Moreover, as is observed for Hirzebruch surfaces, two-stage generalized Bott manifolds are not necessarily isomorphic as varieties even if there is a $c_1$-preserving isomorphism between their integral cohomology rings. In this section, we prove the $c_1$-cohomological rigidity for two-stage {\em Fano} generalized Bott manifolds, that is, we prove  Theorem~\ref{thm_rank_2}.

Let $\B$ be the two-stage generalized Bott manifold associated with a generalized Bott matrix $A$ of type $(n_1,n_2)$: 
\begin{equation}\label{eq:2-stage_GB_mx}
A=\begin{bmatrix}
-1&0\\
\vdots&\vdots\\
-1&0\\
a_1&-1\\
\vdots&\vdots\\
a_{n_2}&-1
\end{bmatrix}_{(n_1+n_2)\times 2}
\end{equation} 
Then 
\begin{equation} \label{eq:2-stage_B}
\B=P(\underline{\C}\oplus\gamma^{a_1}\oplus\cdots \oplus \gamma^{a_{n_2}}),
\end{equation}
where $\gamma$ denotes the tautological line bundle over $\C P^{n_1}$. 
It follows from \eqref{eq_coh_GB} and \eqref{eq:c1} that 
\begin{align} 
H^*(\B;\Z)&=\Z[x_1,x_2]\Big/\left\langle x_1^{n_1+1},x_2\prod_{k=1}^{n_2}(x_2-a_kx_1)\right\rangle,\label{eq:coh_2-stage}\\
c_1(\B)&=\left(n_1+1-\sum_{k=1}^{n_2}a_k\right)x_1+(n_2+1)x_2. \label{eq:c1_2-stage}
\end{align}

We note that the isomorphism class of $\B$ as a variety is unchanged by permuting $a_k$'s, so it depends on the multiset $\{a_1,\dots,a_{n_2}\}$. Related to this, we recall an elementary fact used later. The \emph{elementary symmetric polynomials} in $n$ variables $x_{1},\dots,x_{n}$ for $r=1,\dots,n$ are defined by \[
e_{r}(x_{1},\dots,x_{n})=\sum_{1\leq j_{1}<\dots<j_{r}\leq n} x_{j_{1}}\dots x_{j_{r}}.
\]
For a vector $(b_{1},\dots,b_{n})\in\Z^{n}$, the values $e_r(b_1,\dots,b_n)$ $(1\le r\le n)$ determine the set $\{b_1,\dots,b_n\}$ as a multiset, in particular if $e_r(b_1,\dots,b_{n})=0$ for $1\le r\le n$, then $b_k=0$ for every $k=1,\dots,n$. Indeed, this fact immediately follows from the identity
\[
\prod_{k=1}^n(1+b_kt)=\sum_{r=1}^ne_r(b_1,\dots,b_n)t^r.
\] 

\begin{lemma} \label{lemm:n_1>n_2}
Let $\B$ be the two-stage generalized Bott manifold associated with~\eqref{eq:2-stage_GB_mx}. Suppose that $n_1>n_2$ and there exists a nonzero element $y\in H^2(\B;\Z)$ such that $y^{n_2+1}=0$. Then $a_k=0$ for every $k=1,\dots,n_2$, in particular, $\B$ is isomorphic to $\C P^{n_1}\times \C P^{n_2}$ as a variety. 
\end{lemma}

\begin{proof}
One can express $y=px_1+qx_2$ with integers $p$ and $q$ by \eqref{eq:coh_2-stage}. Then 
\[
0=y^{n_2+1}=(px_1+qx_2)^{n_2+1}=p^{n_2+1}x_1^{n_2+1}+\binom{n_2+1}{1}p^{n_2}qx_1^{n_2}x_2+\dots + q^{n_2+1} x_2^{n_2+1}.
\]
Since $n_2+1\le n_1$, the coefficient of $x_1^{n_2+1}$ above must vanish by \eqref{eq:coh_2-stage}; so $p=0$ and hence $y=qx_2$. Since $y\not=0$ and $y^{n_2+1}=0$ by assumption, we get $x_2^{n_2+1}=0$. Then it follows from \eqref{eq:coh_2-stage} that we have 
\begin{equation*} \label{eq:era}
0= x_2\prod_{k=1}^{n_2}(x_2-a_kx_1)=\sum_{r=1}^{n_2}(-1)^r e_r(a_1,\dots,a_{n_2})x_1^rx_2^{n_2+1-r}.
\end{equation*}
Since $n_1>n_2$ by assumption, one can see from \eqref{eq:coh_2-stage} that the elements $x_1^rx_2^{n_2+1-r}$ $(1\le r\le n_2)$ above are linearly independent; so their coefficients above must vanish. Hence $a_k=0$ for every~$k=1,\dots,n_2$. 
\end{proof}

Let $\B$ (resp. $\tilde \B$) be a two-stage generalized Bott manifold associated with a generalized Bott matrix of type $(n_1,n_2)$ (resp. $(\tilde{n}_1,\tilde{n}_2)$). If their cohomology rings are isomorphic to each other, then $\{n_1,n_2\}=\{\tilde{n}_1,\tilde{n}_2\}$ as sets as remarked in Remark~\ref{rm1}, so $(\tilde{n}_1,\tilde{n}_2)=(n_1,n_2)$ or $(n_2,n_1)$. When their cohomology rings are isomorphic to $H^*(\C P^{n_1}\times \C P^{n_2};\Z)$, both cases can occur but otherwise only the former case occurs as is shown below. 

\begin{lemma} \label{lemm:fiber_dimension}
Let $\B$ and $\tilde{\B}$ and be as above. If $H^*(\B;\Z)\cong H^*(\tilde{\B};\Z)$ as graded rings and they are not isomorphic to $H^*(\C P^{n_1}\times\C P^{n_2};\Z)$, then $(\tilde{n}_1,\tilde{n}_2)=(n_1,n_2)$. 
\end{lemma} 

\begin{proof}
Let $\varphi\colon H^*(\tilde\B;\Z)\to H^*(\B;\Z)$ be an isomorphism as graded rings. Suppose that $(\tilde{n}_1,\tilde{n}_2)\not=(n_1,n_2)$. Then $(\tilde{n}_1,\tilde{n}_2)=(n_2,n_1)$ and $n_1\not=n_2$. We may assume $n_1>n_2$ because otherwise we interchange the role of $\B$ and $\tilde{\B}$. Let $\tilde{x}$ be a nonzero element of $H^2(\tilde{\B};\Z)$ coming from the base space $\C P^{n_2}$ of $\tilde{\B}$, so $\tilde{x}^{n_2+1}=0$. Then since $\varphi(\tilde{x})\not=0$ and $\varphi(\tilde{x})^{n_2+1}=0$, $\B$ is isomorphic to $\C P^{n_1}\times \C P^{n_2}$ as a variety by Lemma~\ref{lemm:n_1>n_2}. This contradicts the assumption on $H^*(\B)$, proving the lemma. 
\end{proof}

For the integers $a_k$ $(1\le k\le n_2)$ in \eqref{eq:2-stage_GB_mx}, we may assume that 
\begin{equation} \label{eq:ai_nonnegative}
a_k\ge 0\quad\text{ for every $k=1,\dots,n_2$}.
\end{equation}
Indeed, $P(E)$ is isomorphic to $P(E\otimes L)$ for any complex vector bundle $E$ and any line bundle $L$; so if the minimum among $a_1,\dots,a_{n_2}$, say $a_d$, is negative, then we consider $P(E\otimes \gamma^{-a_d})$ instead of $P(E)$ for $E=\underline{\C}\oplus\gamma^{a_1}\oplus\cdots \oplus \gamma^{a_{n_2}}$ in \eqref{eq:2-stage_B}, where the exponents $a_k-a_d$ of $\gamma$ appearing in $E\otimes \gamma^{-a_d}$ are all nonnegative. 

The assumption \eqref{eq:ai_nonnegative} is convenient to see the Fano condition for $\B$. As before, we denote the $j$th colomn vector in \eqref{eq:2-stage_GB_mx} by $\vv_j$ where $j=1,2$. Then 
\[
\vv_1+\ve_1^1+\cdots+\ve_1^{n_1}=a_1\ve_2^1+\cdots+a_{n_2}\ve_2^{n_2}\quad\text{and}\quad \vv_2+\ve_2^1+\cdots+\ve_2^{n_2}=\mathbf{0}.
\]
Since $a_k\ge 0$ every $k=1,\dots,n_2$, it follows from Lemma~\ref{lem:Fano} that $\B$ is Fano if and only if 
\begin{equation} \label{eq:Fano_2-stage}
\sum_{k=1}^{n_2}a_k\le n_1.
\end{equation}

We denote the generalized Bott manifold associated with~\eqref{eq:2-stage_GB_mx} by $\B(n_1,(a_1,\dots,a_{n_2}))$, where we take $a_k\ge 0$ for $k=1,\dots,n_2$ by \eqref{eq:ai_nonnegative}. 

\begin{proposition}[cf. {\cite[Proposition~1.8]{McDuff2011}}] \label{prop:2-stage-trivial}
Let $\B=\B(n_1,(a_1,\dots,a_{n_2}))$. Suppose that $H^{\ast}(\B;\Z)$ is isomorphic to $H^{\ast}(\CP^{n_1}\times \CP^{n_2};\Z)$ as graded rings. If either 
\begin{enumerate}
\item $n_1>n_2$, or
\item $n_1\le n_2$ and $\B$ is Fano,
\end{enumerate}
then $a_k=0$ for every $k=1,\dots,n_2$ {\rm (}so $\B$ is isomorphic to $\CP^{n_1}\times \CP^{n_2}$ as a variety{\rm )}. 
\end{proposition}

\begin{proof} 
It follows from~\cite[Theorem~6.1]{CMS-tams} that the assumption $H^{\ast}(\B;\Z)\cong H^{\ast}(\CP^{n_1}\times \CP^{n_2};\Z)$ is equivalent to the existence of an integer $b$ such that 
\begin{equation} \label{eq:height_2}
\prod_{k=1}^{n_2}(1+a_{k}x)=(1+bx)^{n_2+1}\quad\text{in }\ \Z[x]/\langle x^{n_1+1}\rangle.
\end{equation}
If $n_1>n_2$, then we get $b=0$ by comparing the coefficients of the term $x^{n_2+1}$ above. Hence $a_k=0$ for every $k$. If $n_1\leq n_2$ and $\B$ is Fano, then we have $0\leq \sum_{k=1}^{n_2}a_{k}\leq n_1 < n_2+1$ from \eqref{eq:Fano_2-stage}. On the other hand, we have $\sum_{k=1}^{n_2}a_k=(n_2+1)b$ from \eqref{eq:height_2}. Therefore, $b=0$ and hence $\sum_{k=1}^{n_2}a_k=0$. This implies that $a_k=0$ for every $k$ because $a_k\ge 0$ for every $k$. 
\end{proof}

\begin{remark}
\begin{enumerate}
\item The proof above shows that the generalized Bott tower $\B_\bullet$ of height two with $\B=\B_2$ is trivial under the assumption of the proposition. 
\item If the Fano condition is dropped in Proposition~\ref{prop:2-stage-trivial}, then $\B$ is {\em diffeomorphic} to $\C P^{n_1}\times \C P^{n_2}$ (\cite[Corollary 6.3]{CMS-tams}) but not necessarily isomorphic to $\C P^{n_1}\times \C P^{n_2}$ as a variety. Indeed, the Hirzebruch surface $F_a=P(\underline{\C}\oplus\gamma^a)$, which is a two-stage Bott manifold, has the cohomology ring isomorphic to $H^*(\C P^1\times \C P^1;\Z)$ when $a$ is even but $F_a$ is not isomorphic to $\C P^1\times \C P^1$ as a variety unless $a=0$. 
\item When $n_1=n_2=1$, we need the Fano condition as remarked above. However, when $n_1=n_2\ge 2$, the conclusion of the proposition holds without the Fano condition. Indeed, in this case, one can deduce $b=0$ from \eqref{eq:height_2} so that $a_k=0$ for every $k$. 
\end{enumerate}
\end{remark}

Note that $\B(1,(a))$ (i.e. $n_1=n_2=1$) is a Hirzebruch surface $F_a$ and it is Fano if and only if $a=0,1$. When $n_2=1$, the following is known. 

\begin{proposition} \label{prop:n1=1} 
Let $a$ and $\tilde{a}$ be nonnegative integers. 
\begin{enumerate}
\item When $n_1=1$, $H^*(\B(1,(a));\Z)\cong H^*(\B(1,(\tilde{a}));\Z)$ if and only if $a\equiv \tilde{a}\pmod{2}$. 
\item When $n_1>1$, $H^*(\B(n_1,(a));\Z)\cong H^*(\B(n_1,(\tilde{a}));\Z)$ if and only if $a=\tilde{a}$. 
\end{enumerate}

\end{proposition} 

\begin{proof}
(1) is well-known and easy to prove. (2) is Proposition 5.2 in \cite{c-p-s12}. 
\end{proof}

In the cases treated in Propositions~\ref{prop:2-stage-trivial} and~\ref{prop:n1=1}, two-stage {\em Fano} generalized Bott manifolds are distinguished as varieties by their cohomology rings. However, this is not true in general. 

\begin{example}
Let $\B=\B(2,(1,1))$ and $\tilde \B=\B(2,(0,1))$. They are four-dimensional and Fano by \eqref{eq:Fano_2-stage} but not isomorphic as varieties. Indeed, they are ID 70 and ID 141 respectively in the classification list of smooth toric Fano varieties by {\O}bro (\cite{Oeb}), see also \cite[Table 6]{HKM22}. However, there is a graded ring isomorphism 
\[
H^{\ast}(\B;\Z)=\Z[x_{1},x_{2}]/\langle x_{1}^{3},x_{2}(x_{2}-x_{1})^2\rangle \xrightarrow{\varphi}
H^{\ast}(\tilde \B;\Z)=\Z[\tilde x_{1},\tilde x_{2}]/\langle \tilde x_{1}^{3}, \tilde x_{2}^2(\tilde x_{2}- \tilde x_{1})\rangle
\] 
defined by $\varphi(x_1)=\tilde x_1$ and $\varphi(x_2)=\tilde x_1- \tilde x_2$, {so that they are diffeomorphic to each other by~\cite[Theorem 6.1]{CMS-tams}. Note that $c_1(\B)=x_1+3x_2$ and $c_1(\tilde\B)=2\tilde x_1+3\tilde x_2$ and $\varphi$ is not $c_1$-preserving}.
\end{example}

To treat the case $n_2\ge 2$, we recall a lemma. 

\begin{lemma}[{\cite[Lemma 6.2]{CMS-tams}}]\label{lem:lem-CMS}
Assume that $n_2\geq 2$ and {let $(d_1,\dots,d_{n_2})$ be a nonzero integer vector}. If $(a x+b y)^{n_1+1}=0$ in $\Z[x,y]/\langle x^{n_1+1},y\prod_{i=1}^{n_2}(y+d_ix)\rangle$ for some integers $a$ and $b$, then we get $b=0$.
\end{lemma}

Below is our main result in this section.

\begin{proposition}\label{thm:stage-2}
Let $\B=\B(n_1,(a_1,\dots,a_{n_2}))$ and $\tilde \B =\B(n_1,(\tilde a_1,\dots,\tilde a_{n_2}))$. If $\B$ and $\tilde \B$ are Fano and there is a $c_1$-preserving isomorphism between $H^\ast(\B;\Z)$ and $H^\ast(\tilde \B;\Z)$ as graded rings, then $\B$ and $\tilde \B$ are isomorphic as varieties.
\end{proposition}
\begin{proof}
{When $n_2=1$, the theorem follows from Proposition~\ref{prop:n1=1} (in this case, the $c_1$-preserving condition is unnecessary). So, we assume $n_2\ge 2$. Moreover, we may assume that both vectors $(a_1,\dots,a_{n_2})$ and $(\tilde a_1,\dots,\tilde a_{n_2})$ are nonzero by Proposition~\ref{prop:2-stage-trivial}.} We may further assume \eqref{eq:ai_nonnegative} and \eqref{eq:Fano_2-stage} for $a_k$'s and $\tilde{a}_k$'s. Under this situation, we prove $\{a_1,\dots,a_{n_2}\}=\{\tilde a_1,\dots,\tilde{a}_{n_2}\}$ as multisets, which means that $\B$ and $\tilde{\B}$ are isomorphic as varieties. 

We denote by $\tilde{x}_i$ the element in $H^2(\tilde{\B};\Z)$ corresponding to $x_i$ for $i=1,2$. Then $H^*(\tilde{\B};\Z)$ and $c_1(\tilde{\B})$ have the presentation \eqref{eq:coh_2-stage} and \eqref{eq:c1_2-stage} with tilde. 

Let $\varphi\colon H^{\ast}(\tilde\B;\Z) \to H^{\ast}(\B;\Z)$ be a $c_1$-preserving graded ring isomorphism. Since $\varphi(\tilde{x}_1)^{n_1+1}=\varphi(\tilde{x}_1^{n_1+1})=0$ in $H^{\ast}(\B;\Z)$, it follows from Lemma~\ref{lem:lem-CMS} that we have
\begin{equation} \label{eq:varphi_2-stage}
\text{$\varphi(\tilde{x}_1)=\epsilon_{1} x_1$\quad and\quad $\varphi(\tilde x_2)=px_1+\epsilon_{2} x_2$}
\end{equation}
for some integer $p$, where $\epsilon_{1}$ and $\epsilon_{2}$ are $\pm1$ because $\varphi$ is an isomorphism.
Therefore, 
\begin{equation*}
\varphi \left(\tilde x_{2}\prod_{k=1}^{n_2}(\tilde x_{2}-\tilde a_{k}\tilde x_{1}) \right)=(\epsilon_{2} x_{2}+p x_{1})\prod_{k=1}^{n_2}(\epsilon_{2} x_{2}+(p-\tilde a_{k}\epsilon_{1}) x_{1}). 
\end{equation*}
The right hand side above vanishes in $H^{\ast}(\B;\Z)$ because so does the left hand side above by \eqref{eq:coh_2-stage} for $\tilde\B$. It follows from \eqref{eq:coh_2-stage} that there exist a homogeneous polynomial $f(x_{1}, x_{2})$ of degree $n_2-n_1$ when $n_2\ge n_1$ ($f(x_{1},x_{2})=0$ otherwise) and an integer $q$ such that 
\begin{equation} \label{eq-rel}
(\epsilon_{2}x_{2}+p x_{1})\prod_{k=1}^{n_2}(\epsilon_{2}x_{2}+(p-\tilde a_{k}\epsilon_{1})x_{1}) \\
=f(x_{1},x_{2}) x_{1}^{n_1+1}+q x_{2}\prod_{k=1}^{n_2}(x_{2}- a_{k} x_{1})
\end{equation}
as polynomials in ${x}_1$ and ${x}_2$. Comparing the coefficients of $x_{2}^{n_2+1}$ on both sides above, we get
\begin{equation}\label{eq-1}
\epsilon_{2}^{n_2+1}=q.
\end{equation}

On the other hand, since $\varphi$ is $c_{1}$-preserving, it follows from \eqref{eq:c1_2-stage} (for $\B$ and $\tilde\B$) and \eqref{eq:varphi_2-stage} that 
\begin{equation*}
\epsilon_{1}\left(n_1+1-\sum_{k=1}^{n_2}\tilde a_{k}\right) x_{1}+(n_2+1)(p x_{1}+\epsilon_{2} x_{2})=\left((n_1+1)-\sum_{k=1}^{n_2} a_{k}\right) x_{1}+(n_2+1)x_{2}.
\end{equation*}
Comparing the coefficients of ${x}_2$ on both sides above, we get $\epsilon_{2}=1$; so $q=1$ by~\eqref{eq-1} and the identity above reduces to 
\begin{equation}\label{eq-3}
\epsilon_{1}\left(n_1+1-\sum_{k=1}^{n_2}\tilde a_{k}\right)+(n_2+1)p=(n_1+1)-\sum_{k=1}^{n_2} a_{k}. 
\end{equation}
Moreover, comparing the coefficients of $ x_{1}x_{2}^{n_2}$ on both sides of~\eqref{eq-rel} with $\epsilon_2=q = 1$, we get 
\begin{equation}\label{eq-2-2}
(n_2+1)p-\epsilon_{1}\sum_{k=1}^{n_2}\tilde a_{k}=-\sum_{k=1}^{n_2} a_{k}.
\end{equation}
By substituting~\eqref{eq-2-2} into~\eqref{eq-3}, we get $\epsilon_1(n_1+1)=n_1+1$. Therefore, $\epsilon_{1}=1$ and hence \eqref{eq-2-2} reduces to 
\begin{equation} \label{eq:n_2}
(n_2+1)p=\sum_{k=1}^{n_2}\tilde a_{k}-\sum_{k=1}^{n_2} a_{k}.
\end{equation} 

\underline{\textbf{Case 1:} $n_2<n_1$.} In this case, we have $f(x_1,x_2)=0$, so \eqref{eq-rel} with $\epsilon_1=\epsilon_2=q=1$ becomes 
\begin{equation}\label{eq_factorization}
(x_2+px_1)\prod_{k=1}^{n_2}(x_2+(p-\tilde a_k)x_1)=x_2\prod_{k=1}^{n_2}(x_2- a_k x_1)
\end{equation} 
as polynomials in ${x}_1$ and ${x}_2$. Hence 
\begin{equation}\label{eq_p_condition}
p=0 \quad \text{ or } \quad p=\tilde a_{k_0} \quad\text{for some $1\leq k_0\leq n_2$ with $\tilde a_{k_0}>0$}. 
\end{equation}

Suppose that the latter case in \eqref{eq_p_condition} occurs. Then it follows from \eqref{eq_factorization} that $p=- a_{k_1}$ for some $k_1$ ($1\leq k_1\leq n_2$) but this is a contradiction because $\tilde a_{k_0}>0$ while $-{a}_{k_1}\le 0$ by \eqref{eq:ai_nonnegative}. Therefore, the former case in~\eqref{eq_p_condition} must occur, i.e. $p=0$. By substituting $p = 0$ in~\eqref{eq_factorization}, we get
\[
x_2 \prod_{k=1}^{n_2} (x_2 - \tilde a_k x_1)= x_2 \prod_{k=1}^{n_2} (x_2 - a_k x_1).
\]
Therefore, we obtain $\{\tilde a_1,\dots,\tilde{a}_{n_2}\}=\{a_1,\dots,a_{n_2}\}$ as multisets. 

\underline{\textbf{Case 2:} $n_2\ge n_1$.} It follows from ~\eqref{eq:Fano_2-stage} (for $\B$ and $\tilde\B$) that 
\[
\left|\ \sum_{k=1}^{n_2}a_{k}-\sum_{k=1}^{n_2} \tilde a_{k}\ \right| \leq n_1<n_2+1.
\] 
This together with \eqref{eq:n_2} implies $p=0$, so \eqref{eq-rel} with $p=0$ and $\epsilon_{1}=\epsilon_{2}=q=1$ becomes
\begin{equation*}
x_{2}\prod_{k=1}^{n_2}( x_{2}-\tilde a_{k} x_{1}) =f(x_{1},x_{2})x_{1}^{n_1+1}+x_{2}\prod_{k=1}^{n_2}(x_{2}- a_{k}x_{1})
\end{equation*}
as polynomials in ${x}_1$ and ${x}_2$. Comparing the coefficients of $x_{1}x_{2}^{n_2},x_{1}^{2}x_{2}^{n_2-1},\dots,x_{1}^{n_1}x_{2}^{n_2-n_1+1}$ on both sides above, we get
\begin{equation} \label{eq:a&tildea}
e_{k}( \tilde a_{1},\dots,\tilde a_{n_2})=e_{k}(a_{1},\dots,a_{n_2})\qquad\text{for }k=1,\dots,n_1.
\end{equation}
Since $a_k$'s are nonnegative integers by \eqref{eq:ai_nonnegative} and $\sum_{k=1}^{n_2}a_k\le n_1$ by \eqref{eq:Fano_2-stage}, at most $n_1$ elements in $\{a_1,\dots,a_{n_2}\}$ are nonzero. The same is true for $\tilde{a}_k$'s. Therefore, we get $\{\tilde a_1,\dots,\tilde{a}_{n_2}\}=\{a_1,\dots,a_{n_2}\}$ as multisets by \eqref{eq:a&tildea}. 
\end{proof}

Now we are ready to prove Theorem~\ref{thm_rank_2}.

\begin{proof}[Proof of Theorem~\ref{thm_rank_2}]
Any smooth compact toric variety of Picard number two is a two-stage generalized Bott manifold (\cite{klei88}).  So, suppose that $X$ and $Y$ are two-stage Fano generalized Bott manifolds and there is a $c_1$-preserving graded ring isomorphism between their integral cohomology rings.  By Lemma~\ref{lemm:fiber_dimension} and Proposition~\ref{prop:2-stage-trivial}, we may assume that $X$ and $Y$ are associated with generalized Bott matrices of the same type. Then $X$ and $Y$ are isomorphic as varieties by Proposition~\ref{thm:stage-2}. 
\end{proof}

\section{Related cohomological rigidity} \label{sec:concluding_remarks}

In this section, we overview related cohomological rigidity problems and results. 
All cohomology groups are taken with $\Z$ coefficients unless otherwise stated. 

\subsection{Equivariant cohomology and equivariant first Chern class.}\label{subsection_equivariant_cohomologies} Let $X$ be a smooth compact toric variety and $\T$ the algebraic torus acting on $X$. The equivariant cohomology of $X$ is defined as 
\[
H^*_{\T}(X):=H^*(E{\T}\times_{\T} X)
\]
where $E\T\to B\T$ is the universal principal $\T$-bundle and $E{\T}\times_{\T} X$ denotes the orbit space of $E{\T}\times X$ by the diagonal $\T$-action. 
The equivariant cohomology $H^*_{\T}(X)$ is not only a ring but also an algebra over $H^*(B\T)$ through the projection $E{\T}\times_{\T} X\to B\T$. 

The group ${\rm Aut}(X)$ of all automorphisms of $X$ is known to be an algebraic group of finite dimension and the algebraic torus $\T$ acting on $X$ determines a maximal torus of ${\rm Aut}(X)$ (see \cite[Section 3.4]{oda88}). This implies that if smooth compact toric varieties $X$ and $Y$ are isomorphic as varieties, then they are isomorphic as toric varieties up to an automorphism of $\T$, that is, there is an isomorphism $f\colon X\to Y$ together with a group automorphism $\sigma$ of $\mathbb{T}$ such that $f(gx)=\sigma(g)f(x)$ for $g\in \T$ and $x\in X$. Therefore, if $X$ and $Y$ are isomorphic as varieties, then $H^*_\T(X)$ and $H^*_\T(Y)$ are weakly isomorphic as algebras over $H^*(B\T)$, which means that there is a ring isomorphism $\Phi\colon H^*_\T(Y)\to H^*_\T(X)$ together with an automorphism $\sigma$ of $\T$ such that $\Phi(u\alpha)=\sigma^*(u)\Phi(\alpha)$ for any $u\in H^*(B\T)$ and $\alpha\in H^*_\T(Y)$, where $\sigma^*$ denotes the automorphism of $H^*(B\T)$ induced from~$\sigma$. Moreover, the ring isomorphism $\Phi$ induced from the variety isomorphism $f$ preserves the equivariant first Chern classes of $X$ and $Y$. 

It turns out that the converse holds, namely if there is a weak $H^*(B\T)$-algebra isomorphism $\Phi\colon H^*_\T(Y)\to H^*_\T(X)$ preserving the equivariant first Chern classes of $X$ and $Y$, then $X$ and $Y$ are isomorphic as varieties. (Note. It is pointed out in \cite[Remark 2.5]{HKM22} that the condition preserving the equivariant first Chern classes is necessary for \cite[Theorem 1.1]{Ma08}.) Such $\Phi$ induces an isomorphism $\varphi$ between $H^*(X)$ and $H^*(Y)$ preserving the first Chern classes of $X$ and $Y$. Therefore, Conjecture~1.1 suggests that it might be possible to recover $\Phi$ from $\varphi$ for smooth toric Fano varieties. 

\subsection{Cohomological rigidity over a commutative ring $\Lambda$} The cohomological rigidity problem posed in \cite{ma-su08} asks whether smooth compact toric varieties are diffeomorphic (or homeomorphic) if they have isomorphic integral cohomology rings. Many partial positive results are known, but no counterexample is known. To state known results for the cohomological rigidity problem, it is convenient to fix a family $\mcal{M}$ of smooth manifolds and say that $\mcal{M}$ is {\em cohomologically rigid} if any two objects in $\mcal{M}$ are distinguished up to diffeomorphism (or homeomorphism) by their integral cohomology rings. The family of $2$-stage generalized Bott manifolds is cohomologically rigid (\cite{CMS-tams}). However, it is not known whether the family of generalized Bott manifolds is cohomologically rigid. A recent notable achievement is that the family of Bott manifolds is cohomologically rigid (\cite{c-h-j}). See \cite{ch-pa16, c-p-s12, HKMP20} for further results. 

For a real analogue of smooth compact toric manifolds, 
such as real loci of compact smooth toric varieties or small covers introduced by Davis--Januszkiewicz (see \cite{BP_toric_topology, DavisJanuszkiewicz91}),  
it is natural to take cohomology rings with $\Z/2\Z$-coefficients. We say that a family $\mcal{M}$ of smooth manifolds is \emph{cohomologically rigid over a commutative ring $\Lambda$} if any two objects in $\mcal{M}$ are distinguished up to diffeomorphic (or homeomorphic) by their cohomology rings with $\Lambda$-coefficients. It is known that the family of real Bott manifolds is cohomologically rigid over $\Z/2\Z$ (\cite{c-m-o17, ka-ma09}). Real Bott manifolds are real loci of Bott manifolds and provide examples of Riemannian flat manifolds. Similarly, the family of hyperbolic $3$-manifolds of L\"obel type, which are small covers over $3$-dimensional right-angled hyperbolic polytopes, is also cohomologically rigid over $\Z/2\Z$ (\cite{b-e-m-p-p17}). However, the family of $2$-stage real generalized Bott manifolds is not cohomologically rigid over $\Z/2\Z$ (\cite{masu10}) although the family of $2$-stage generalized Bott manifolds is cohomologically rigid over $\Z$ as is mentioned above. 

\subsection{Cohomological super-rigidity} One can also think of an algebraic version of the cohomological rigidity. Following \cite{HKMP20}, we may say that a family $\mcal{V}$ of smooth algebraic varieties is {\em cohomologically super-rigid} if any objects in~$\mcal{V}$ are distinguished up to isomorphism by their integral cohomology rings. Propositions~\ref{prop:2-stage-trivial} and~\ref{prop:n1=1} are results of this type, see also \cite{HKMP20} for results of this type.

\bibliographystyle{amsplain}

\end{document}